\documentclass{article}
\usepackage{amsmath}
\usepackage{bm}
\usepackage{amssymb}
\usepackage{graphicx}
\usepackage{natbib}
\usepackage[a4paper,left=3cm,right=3cm,top=3cm,bottom=3cm]{geometry}
\usepackage{fancyhdr}
\newtheorem{theorem}{Theorem}[section]
\newtheorem{corollary}{Corollary}[section]
\newtheorem{proposition}{Proposition}[section]
\newtheorem{lemma}{Lemma}[section]
\newtheorem{definition}{Definition}[section]

\makeatletter % `@' now normal "letter"
\@addtoreset{equation}{section}
\makeatother  % `@' is restored as "non-letter"

\newenvironment{remark}{{\noindent\sc Remark.}\quad}{~}

\newenvironment{keywords}{{\noindent\bf Keywords.}}{~}
\newenvironment{AMS}{{\noindent\bf AMS subject classifications.}}{~}
\newenvironment{proof}{{\noindent\it Proof.}\quad}{\hfill $\square$\\}

\title{\textsc{Regularized Weighted Discrete Least Squares Approximation Using Gauss Quadrature Points}}

\author{Congpei An\footnotemark[1]
        \quad and\quad Hao-Ning Wu\footnotemark[2]}

\date{}
\begin{document}

\maketitle

\renewcommand{\thefootnote}{\fnsymbol{footnote}}
\footnotetext[1]{School of Economic Mathematics, Southwestern University of Finance and Economics, Chengdu 611130, China.\\ \indent~~ Email address: andbach1984@hotmail.com, ancp@swufe.edu.cn}
\footnotetext[2]{Department of Mathematics, The University of Hong Kong, Hong Kong, China. \\ \indent~~ Email address: haoning.wu@outlook.com, hnwu@hku.hk}
%\footnotetext[3]{Second author's address.}

\begin{abstract}
We consider polynomial approximation over the interval $[-1,1]$ by regularized weighted discrete least squares methods with $\ell_2-$ or $\ell_1-$regularization, respectively. As the set of nodes we use Gauss quadrature points (which are zeros of orthogonal polynomials). The number of Gauss quadrature points is $N+1$. For $2L\leq2N+1$, with the aid of Gauss quadrature, we obtain approximation polynomials of degree $L$ in closed form without solving linear algebra or optimization problems. In fact, these approximation polynomials can be expressed in the form of the barycentric interpolation formula \citep{berrut2004barycentric} when an interpolation condition is satisfied.
We then study the approximation quality of the $\ell_2-$regularized approximation polynomial in terms of Lebesgue constants, and the sparsity of the $\ell_1-$regularized approximation polynomial. Finally, we give numerical examples to illustrate these theoretical results and show that a well-chosen regularization parameter can lead to good performance, with or without contaminated data.
\end{abstract}



\section{Introduction}
\label{sec;introduction}
In this paper, we are interested in approximating or recovering a function (possibly noisy) $f\in\mathcal{C}([-1,1])$ by a polynomial
\begin{equation}\label{solutionp}
p_{L}(x)=\sum_{\ell=0}^L\beta_{\ell}\tilde{\Phi}_{\ell}(x)\in \mathbb{P}_L,\quad x\in[-1,1],
\end{equation}
where $\mathbb{P}_L$ is a linear space of polynomials of degree at most $L$, and $\tilde{\Phi}_{\ell}(x)$, $\ell=0,1,\ldots,L$, are normalized orthogonal polynomials \citep{von1939orthogonal,gautschi2004orthogonal}.
As long as the basis $\{\tilde{\Phi}_{\ell}(x)\}_{\ell=0}^L$ for $\mathbb{P}_L$ is given, the next step is to determine coefficients $\beta_{\ell},\,\ell=0,\ldots,L$.
We will consider the $\ell_2-$regularized approximation problem
\begin{equation}\label{problem:l2regu}
\min_{\beta_{\ell}\in\mathbb{R}}~~ \left\{\sum_{j=0}^N\omega_j\left(\sum_{\ell=0}^{L}\beta_{\ell}\tilde{\Phi}_{\ell}(x_j)-f(x_j)\right)^2+\lambda\sum_{\ell=0}^L(\mu_{\ell}\beta_{\ell})^2\right\},\quad\lambda>0,
\end{equation}
and the $\ell_1-$regularized approximation problem
\begin{equation}\label{problem:l1regu}
\min_{\beta_{\ell}\in\mathbb{R}}~~ \left\{\sum_{j=0}^N\omega_j\left(\sum_{\ell=0}^{L}\beta_{\ell}\tilde{\Phi}_{\ell}(x_j)-f(x_j)\right)^2+\lambda\sum_{\ell=0}^L|\mu_{\ell}\beta_{\ell}|\right\},\quad\lambda>0,
\end{equation}
where $f$ is a given continuous function with values (possibly noisy) taken at $N+1$ distinct points $x_0,x_1,\ldots,x_N$ over the interval $[-1,1]$; $2L\leq2N+1$; ${\rm \bf{w}}=[\omega_0,\omega_1,\ldots,\omega_N]^T$ is a vector of positive Gauss quadrature weights \citep{gautschi2004orthogonal}; $\{\mu_{\ell}\}_{\ell=0}^L$ is a nonnegative nondecreasing sequence, which penalizes coefficients $\{\beta_{\ell}\}_{\ell=0}^L$; and $\lambda>0$ is the regularization parameter.

It is known that approximation schemes \eqref{problem:l2regu} and \eqref{problem:l1regu} are special cases of classical penalized least squares methods, see \citet{powell1967maximum}, \citet{von1990data}, \citet{gautschi2004orthogonal}, \citet{lazarov2007balancing}, \citet{cai2009split}, \citet{kim2009ell_1}, \citet{an2012regularized}, \citet{xiang2013regularization} and \citet{zhou2018spherical}. Some optimization methods or iterative algorithms are presented to find minimizers. However, we will concentrate on the aspect of constructing minimizers to problems \eqref{problem:l2regu} and \eqref{problem:l1regu} by means of orthogonal polynomials and  Gauss quadrature  \citep{kress1998numerical,gautschi2004orthogonal,gautschi2011numerical,trefethen2013approximation}.
In this paper, Gauss quadrature will play an important role.
 We assume that the weight function $w:~(-1,1)\rightarrow\mathbb{R}$ is positive, such that
$$\int_{-1}^1w(x)dx<\infty,\quad \int_{-1}^1x^iw(x)dx<\infty,\,i=1,2,\ldots.$$
\begin{definition}\label{def:gauss}
A quadrature formula
\begin{equation*}
\int_{-1}^1w(x)f(x)dx\approx\sum\limits_{j=0}^N\omega_jf(x_j)
\end{equation*}
with $N+1$ distinct quadrature points $x_0,x_1,\ldots,x_N$ is called a \emph{Gauss quadrature formula} if it integrates all polynomials $p\in\mathbb{P}_{2N+1}$ exactly, i.e., if
\begin{equation}\label{gauss}
\sum\limits_{j=0}^N\omega_jp(x_j)=\int_{-1}^1w(x)p(x)dx\quad\forall p\in\mathbb{P}_{2N+1}.
\end{equation}
$x_0,x_1,\ldots,x_N$ are called \emph{Gauss quadrature points}.
\end{definition}
\noindent Without causing confusion, we call $\mathcal{X}_{N+1}=\{x_j\}_{j=0}^N$ Gauss quadrature points. Throughout this paper, we always assume that $\mathcal{X}_{N+1}$ are Gauss quadrature points. It is well known (see, for example, \citet{powell1981approximation}, \citet{kress1998numerical} and \citet{gautschi2011numerical}) that Gauss quadrature points $\mathcal{X}_{N+1}$ are the zeros of the orthogonal polynomial of degree $N+1$. The orthogonality is with respect to the $L_2$ inner product
\begin{equation*}
(f,g)_{L_2}:=\int_{-1}^1w(x)f(x)g(x)dx.
\end{equation*}
This inner product induces the standard $L_2$ norm
\begin{equation*}
\|f\|_{L_2}:=\sqrt{(f,f)_{L_2}}=\left(\int_{-1}^1|f(x)|^2dx\right)^{1/2}.
\end{equation*}

Given a continuous function $f$ defined on $[-1,1]$, sampling on $\mathcal{X}_{N+1}$ generates
\begin{equation*}
{\rm \bf{f}}:={\rm \bf{f}}(\mathcal{X}_{N+1})=[f(x_0),f(x_1),\ldots.f(x_N)]^T\in\mathbb{R}^{N+1}.
\end{equation*}
Let ${\rm \bf{A}}_L:={\rm \bf{A}}_L(\mathcal{X}_{N+1})\in\mathbb{R}^{(N+1)\times(L+1)}$ be a matrix of orthogonal polynomials evaluated at $\mathcal{X}_{N+1}$:
\begin{equation*}
{\rm \bf{A}}_L=\left[\tilde{\Phi}_{\ell}(x_j)\right]\in\mathbb{R}^{(N+1)\times(L+1)},\quad j=0,1,\ldots,N,\quad\ell=0,1,\ldots,L.
\end{equation*}
By subtracting the structure \eqref{solutionp} of the approximation polynomial into the $\ell_2-$regularized approximation problem \eqref{problem:l2regu}, the problem \eqref{problem:l2regu} transforms into the following problem
\begin{equation}\label{problem:l2regumatrix}
\underset{{\bm \beta}\in \mathbb{R}^{L+1}}{\min}~~\|{\bf{W}}^{\frac12}({\rm \bf{A} }_L{\bm{\beta}}-{\rm \bf{f}})\|^2_2+\lambda\|{\rm \bf{R}}_L{\bm \beta}\|_2^2,\quad\lambda>0,
\end{equation}
where
\begin{equation*}
{\bf{W}}={\rm diag}(\omega_0,\omega_1,\ldots,\omega_N)\in\mathbb{R}^{(N+1)\times(N+1)},
\end{equation*}
and the diagonal matrix ${\rm \bf{R}}_L:={\rm diag}(\mu_0,\mu_1,\ldots,\mu_L)\in\mathbb{R}^{(L+1)\times(L+1)}$ is a semi-definite positive matrix.

With the same basis and weight vector as the $\ell_2-$regularized approximation problem \eqref{problem:l2regu} above, the $\ell_1-$regularized approximation problem \eqref{problem:l1regu} transforms into
\begin{equation}\label{problem:l1regumatrix}
\underset{{\bm \beta}\in \mathbb{R}^{L+1}}{\min}~~\|{\bf{W}}^{\frac12}({\rm \bf{A} }_L^T{\bm \beta}-{\rm \bf{f}})\|^2_2+\lambda\|{\rm \bf{R}}_L{\bm \beta}\|_1,\quad\lambda>0.
\end{equation}
Now the next step is to fix $\bm{\beta}=[\beta_0,\beta_1,\ldots,\beta_L]^T\in\mathbb{R}^{L+1}$.

The goal of this paper is to construct approximation polynomials in the form of \eqref{solutionp}. We specify coefficients $\{\beta_{\ell}\}_{\ell=0}^L$ by solving problems \eqref{problem:l2regumatrix} and \eqref{problem:l1regumatrix} directly. This immediately yields the $\ell_2-$\textit{regularized approximation polynomial} $p_{L,N+1}^{\ell_2}$ \eqref{p:l2} and the $\ell_1-$\textit{regularized approximation polynomial} $p_{L,N+1}^{\ell_1}$ \eqref{p:l1}.

As is known to all, Gauss quadrature goes hand in hand with the theory and computation of orthogonal polynomials, see \citet{gautschi2004orthogonal} and \citet{trefethen2013approximation} and references therein. Orthogonal polynomials occur in a wide range of applications and act as a remarkable role in pure and applied mathematics. Chebyshev polynomials and Legendre polynomials are two excellent factors in the family of orthogonal polynomials. Many polynomial approximation textbooks introduce fruitful results of Chebyshev and Legendre polynomials \citep{von1939orthogonal,powell1981approximation,gautschi2011numerical,trefethen2013approximation}. In particular, we take these two orthogonal polynomials (Chebyshev and Legendre) as representative examples in the choice of basis and Gauss quadrature points.

In the next section, we introduce some necessary notations and terminologies. The construction of the $\ell_2-$ and $\ell_1-$regularized minimizers to problems \eqref{problem:l2regu} and \eqref{problem:l1regu} are presented, respectively. The crucial fact is that both $p_{L,N+1}^{\ell_2}$ and $p_{L,N+1}^{\ell_1}$ could be presented in the barycentric form under an interpolation condition, see the $\ell_2-$regularized barycentric interpolation formula \eqref{L2regubary} and the $\ell_1-$regularized barycentric interpolation formula \eqref{L1regubary}. It is worth noting that the Wang-Xiang formula \citep{wang2012convergence} is a special case of $p_{L,N+1}^{\ell_2}$ when we set Legendre polynomials as the basis, see Section \ref{barycentricform}. In Section \ref{l2quality}, we study the quality of the approximation polynomial $p_{L,N+1}^{\ell_2}$ in terms of Lebesgue constants. We illustrate that Lebesgue constants decay when the regularization parameter increases. Section \ref{l1sparsity} analyses the $\ell_1-$regularized approximation problem \eqref{problem:l1regu} in the view of sparsity. In particular, we derive a sharp upper bound of nonzero entries in the solution to the $\ell_1-$regularized approximation problem \eqref{problem:l1regumatrix}. We consider, in Section \ref{numericalexperiments}, numerical experiments containing approximation with exact and contaminated data.

All numerical results\footnote{All codes are available at \\
\indent\indent https://github.com/HaoNingWu/Regularized-Least-Squares-Approximation-using-Orthogonal-Polynomials.} in this paper are carried out by using MATLAB R2017A on a desktop (8.00 GB RAM, Intel(R) Processor 5Y70 at 1.10 GHz and 1.30 GHz) with the Windows 10 operating system.

\section{Regularized weighted least squares approximation}
The construction of minimizers to problems \eqref{problem:l2regu} and \eqref{problem:l1regu} is presented in this section.
\subsection{$\ell_2-$regularized approximation problem}\label{sec:l2}
First, we consider solving the $\ell_2$-regularized weighted discrete least squares approximation problem (the \textit{$\ell_2$-regularized approximation problem}) \eqref{problem:l2regu}. The problem can be transformed into a convex and differential optimization problem \eqref{problem:l2regumatrix}.

Taking the first derivative of the objective function in the $\ell_2-$regularized approximation problem \eqref{problem:l2regumatrix} in the matrix form with respect to $\bm{\beta}$ leads to the first order condition
\begin{equation}\label{wproblem:l22w}
\left({\rm \bf{A}}_L^T{\bf{W}}{\rm \bf{A}}_L+\lambda{\rm \bf{R}}_L^T{\rm \bf{R}}_L\right){\bm \beta}={\rm \bf{A}}_L^T{\bf{W}}{\rm \bf{f}},~~~~\lambda>0.
\end{equation}
One may solve the first order condition \eqref{wproblem:l22w} using methods of numerical linear algebra; however, in this paper we concentrate on how to obtain the solution to the first order condition \eqref{wproblem:l22w} in a closed form.
\begin{lemma}\label{thm:diagonal}
Let $\{\tilde{\Phi}_{\ell}\}_{\ell=0}^L$ be a class of normalized orthogonal polynomials with the weight function $w(x)$ over $[-1,1]$, and $\mathcal{X}_{N+1}=\{x_0,x_1,\ldots,x_N\}$ be the set of zeros of $\tilde{\Phi}_{N+1}$. Assume $2L\leq2N+1$ and ${\rm \bf{w}}$ is a vector of weights satisfying the Gauss quadrature formula \eqref{gauss}. Then
\begin{equation*}
{\rm \bf{H}}_L:={\rm \bf{A}}_L^T{\bf{W}}{\rm \bf{A}}_L={\rm \bf{I}}_{L}\in\mathbb{R}^{(L+1)\times(L+1)},
\end{equation*}
where ${\rm\bf{I}}_L$ is the identity matrix.
\end{lemma}
\begin{proof}
By the structure of the matrix ${\rm \bf{H}}_L$ and the definition of Gauss quadrature formula (see \eqref{gauss}), we obtain
\begin{equation*}\begin{split}
({\rm \bf{H}}_L)_{\ell,\ell'} = \sum_{j=0}^N\omega_j\tilde{\Phi}_{\ell}(x_j)\tilde{\Phi}_{\ell'}(x_j)= \int_{-1}^1w(x)\tilde{\Phi}_{\ell}(x)\tilde{\Phi}_{\ell'}(x)dx=\delta_{\ell,\ell'},
\end{split}
\end{equation*}
where $\delta_{\ell,\ell'}$ is the Kronecker delta. The middle equality holds from $\tilde{\Phi}_{\ell}(x)\tilde{\Phi}_{\ell'}(x)\in\mathbb{P}_{2L}\subset\mathbb{P}_{2N+1}$, and the last equality holds because of the orthonormality of $\{\tilde{\Phi}_{\ell}\}_{\ell=0}^L$.
\end{proof}
\begin{theorem}\label{thm:l2}
Under the condition of Lemma \ref{thm:diagonal}, the optimal solution to the $\ell_2-$regularized approximation problem \eqref{problem:l2regumatrix} in the matrix form can be expressed by
\begin{equation}\label{beta:l2}
\beta_{\ell}=\frac{1}{1+\lambda\mu_{\ell}^2}\sum_{j=0}^N\omega_j\tilde{\Phi}_{\ell}(x_j)f(x_j),\quad \ell=0,1,\ldots,L,\quad\lambda>0.
\end{equation}
Consequently, the $\ell_2-$regularized approximation polynomial is
\begin{equation}\label{p:l2}
p_{L,N+1}^{\ell_2}(x) =\sum_{\ell=0}^L\frac{\tilde{\Phi}_{\ell}(x)}{1+\lambda\mu_{\ell}^2}\sum_{j=0}^N\omega_j\tilde{\Phi}_{\ell}(x_j)f(x_j).
\end{equation}
\end{theorem}
\begin{proof}
This is immediately obtained from the first order condition \eqref{wproblem:l22w} of the problem \eqref{problem:l2regumatrix} and Lemma \ref{thm:diagonal}.
\end{proof}

In the limiting case $N\rightarrow\infty$, we obtain the following simple but significant result.
\begin{theorem}\label{l2convergence}
Adopt the notation and assumptions of Lemma \ref{thm:diagonal}.
Let $f\in{\mathcal{C}}([-1,1])$, and let $L\geq0$ be given. Then the polynomial $p_{L,N+1}^{\ell_2}$ \eqref{p:l2} has the uniform limit $p_L^*$ as $N\rightarrow\infty$, that is
\begin{equation*}
\lim\limits_{N\rightarrow\infty}\|p_{L,N+1}^{\ell_2}-p_L^*\|_{\infty}=0,
\end{equation*}
where \begin{equation}\label{cbeta:l2}
p_L^*(x)=\sum\limits_{\ell=0}^L\frac{\tilde{\Phi}_{\ell}(x)}{1+\lambda\mu_{\ell}^2}\int_{-1}^1w(x)\tilde{\Phi}_{\ell}(x)f(x)dx.
\end{equation}
\end{theorem}
\begin{proof}
Since the interval $[-1,\,1]$ is a compact set, and since the sums over $\ell$ in \eqref{beta:l2} and \eqref{cbeta:l2} are finite, to prove the theorem, it is sufficient to prove that
\begin{equation}\label{limit}
\lim\limits_{N\rightarrow\infty}\sum_{j=0}^N\omega_j\tilde{\Phi}_{\ell}(x_j)f(x_j)=\int_{-1}^1w(x)\tilde{\Phi}_{\ell}(x)f(x)dx,\quad 0\leq\ell\leq L.
\end{equation}
Since $w(x)\tilde{\Phi}_{\ell}(x)f(x)\in \mathcal{C}([-1,\,1])$, the result follows from that, the sequence of Gauss quadrature formulae is convergent \citep[Chapter 9]{kress1998numerical}. Hence \eqref{limit} holds, proving the whole theorem.
\end{proof}

\subsection{$\ell_1-$regularized approximation problem}\label{sec:l1}
Now we discuss the $\ell_1-$regularized approximation problem \eqref{problem:l1regu}, but we still convert it into solving the problem \eqref{problem:l1regumatrix} in a matrix form.
To solve this problem, we first define the soft threshold operator $S_k(a)$.
\begin{definition}[\citealp{donoho1994ideal}]
The \textit{soft threshold operator}, denoted as $S_k(a)$, is defined as
\begin{equation*}\label{equ:soft}
S_k(a)=\max(0,a-k)+\min(0,a+k).
\end{equation*}
\end{definition}
\begin{theorem}\label{prop:l1}
Adopt the notation and assumptions of Lemma \ref{thm:diagonal}. Then the $\ell_1-$regularized approximation problem \eqref{problem:l1regumatrix} in the matrix form has the unique closed-form solution
\begin{equation}\label{beta:l1}
\beta_{\ell}=\frac{1}{2}S_{\lambda\mu_{\ell}}(2\alpha_{\ell}),\quad \ell=0,1,\ldots,L,
\end{equation}
where $\alpha_{\ell}=\sum_{j=0}^N\omega_j\tilde{\Phi}_{\ell}(x_j)f(x_j)$. Consequently, the $\ell_1-$regularized approximation polynomial is
\begin{equation}\label{p:l1}
p_{L,N+1}^{\ell_1}(x)=\frac{1}{2}\sum\limits_{\ell=0}^N{S_{\lambda\mu_{\ell}}\left(2\alpha_{\ell}\right)}\tilde{\Phi}_{\ell}(x).
\end{equation}
\end{theorem}
The method of the proof is similar to the proof of \citet[Theorem 5.1]{zhou2018spherical}. But we explain that our regularized least squares approximation problem \eqref{problem:l1regu} is over the interval $[-1,1]$ rather than over the unit sphere.

\begin{proof}
Since ${\rm \bf{H}}_L$ is non-singular, for the $\ell_1-$regularized approximation problem \eqref{problem:l1regumatrix} in the matrix form, its first order condition is
\begin{equation}\label{l1deri}
{\bf{0}}\in2{\rm \bf{H}}_L\bm{\beta}-2{\rm \bf{A}}_L^T{\bf{W}}{\rm \bf{f}}+\lambda\partial(\|{\rm \bf{R}}_L\bm{\beta}\|_1),
\end{equation}
where $\partial(\cdot)$ denotes the subgradient \citep{clarke1990optimization}. Since ${\rm \bf{H}}_{L}={\rm \bf{I}}_{L}$ is an identity matrix and ${\rm \bf{R}}_L$ is a diagonal matrix, $\bm{\beta}$ is the solution to the $\ell_1-$regularized approximation problem \eqref{problem:l1regumatrix} in the matrix form if and only if
\begin{equation}\label{equ:optimal}
0\in2\beta_{\ell}-2\alpha_{\ell}+\lambda\mu_{\ell}\partial|\beta_{\ell}|,\quad \ell=0,1,\ldots,L,
\end{equation}
where $-1\leq\partial|\beta_{\ell}|\leq1$. Let $\beta_{\ell}^*$ be the optimal solution to the problem \eqref{equ:optimal}, then
\begin{equation*}
\beta_{\ell}^*=\frac{1}{2}(2\alpha_{\ell}-\lambda\mu_{\ell}\partial|\beta_{\ell}^*|),\quad \ell=0,1,\ldots,L.
\end{equation*}
Next there exist three cases to be considered:
\begin{enumerate}
  \item If $2\alpha_{\ell}>\lambda\mu_{\ell}$, then $2\alpha_{\ell}-\lambda\mu_{\ell}\partial|\beta_{\ell}^*|>0$, thus $\beta_{\ell}^*>0$, yielding $\partial|\beta_{\ell}^*|=1$, then
      \begin{equation*}
      \beta_{\ell}^*=\frac{1}{2}(2\alpha_{\ell}-\lambda\mu_{\ell})>0.
      \end{equation*}
  \item If $2\alpha_{\ell}<-\lambda\mu_{\ell}$, then $2\alpha_{\ell}+\lambda\mu_{\ell}\partial|\beta_{\ell}^*|<0$, thus $\beta_{\ell}^*<0$, giving $\partial|\beta_{\ell}^*|=-1$, then
      \begin{equation*}
      \beta_{\ell}^*=\frac{1}{2}(2\alpha_{\ell}+\lambda\mu_{\ell})<0.
      \end{equation*}
  \item If $-\lambda\mu_{\ell}\leq2\alpha_{\ell}\leq\lambda\mu_{\ell}$, then on the one hand, $\beta_{\ell}^*>0$ leads to $\partial|\beta_{\ell}|=1$, and thus $\beta_{\ell}^*\leq0$, on the other hand, $\beta_{\ell}^*<0$ produces $\partial|\beta_{\ell}|=-1$, and hence $\beta_{\ell}^*\geq0$. Thus
      \begin{equation*}
      \beta_{\ell}^*=0.
      \end{equation*}
\end{enumerate}
As we hoped, with the aid of soft threshold operator, we obtain
\begin{equation*}\begin{split}
\beta_{\ell}^* &= \frac{1}{2}(\max(0,2\alpha_{\ell}-\lambda\mu_{\ell})+\min(0,2\alpha_{\ell}+\lambda\mu_{\ell}))\\
&=\frac{1}{2}S_{\lambda\mu_{\ell}}(2\alpha_{\ell}).
\end{split}
\end{equation*}
\end{proof}

\subsection{Regularized barycentric interpolation formulae}\label{barycentricform}
In this subsection, we focus on the condition of $L=N$ and the required interpolation conditions
\begin{equation*} p(x_j)=f(x_j),\quad j=0,1,\ldots,N ,\end{equation*}
where $p\in\mathbb{P}_L$ is the interpolant of $f$. As pointed out by \citet*{wang2014explicit}, ``barycentric interpolation is arguably the method of choice for numerical polynomial interpolation". It is significant to express the $\ell_2-$regularized approximation polynomial \eqref{p:l2} and the $\ell_2-$regularized approximation polynomial \eqref{p:l1} in a barycentric form \citep{berrut2004barycentric,higham2004numerical}
\begin{equation*}
p(x)=\dfrac{\sum\limits_{j=0}^N\dfrac{\Omega_j}{x-x_j}f(x_j)}{\sum\limits_{j=0}^N\dfrac{\Omega_j}{x-x_j}},\quad \Omega_j=\frac{1}{\prod_{k\neq j}(x_k-x_j)},
\end{equation*}
respectively. $\{\Omega_j\}_{j=0}^N$ are called barycentric weights.
The study of the $\{\Omega_j\}_{j=0}^N$ for roots and extrema of classical orthogonal polynomials is well developed, see \citet{salzer1972lagrangian}, \citet{schwarz1989numerical}, \citet{berrut2004barycentric}, \citet{wang2012convergence} and \citet{wang2014explicit}.

We first derive the $\ell_2-$regularized barycentric interpolation formula. Recall \eqref{p:l2}, we have
\begin{align}\label{minimizer1}
p_{L,N+1}^{\ell_2}(x) &= \sum\limits_{\ell=0}^N\dfrac{\sum_{j=0}^N\omega_j\tilde{\Phi}_{\ell}(x_j)f(x_j)}{1+\lambda\mu_{\ell}^2}\tilde{\Phi}_{\ell}(x) \notag \\
&=\sum_{j=0}^N\omega_jf(x_j)\sum_{\ell=0}^N\dfrac{\tilde{\Phi}_{\ell}(x_j)\tilde{\Phi}_{\ell}(x)}{1+\lambda\mu_{\ell}^2}.
\end{align}
From the orthonormality of $\tilde{\Phi}_{\ell}(x)$, $\ell=0,1,\ldots,N$, by letting $f(x)\equiv1$ we have
\begin{equation*}
\begin{split}
\sum\limits_{j=0}^N\omega_j\sum_{\ell=0}^N\dfrac{\tilde{\Phi}_{\ell}(x_j)\tilde{\Phi}_{\ell}(x)}{1+\lambda\mu_{\ell}^2}
&=\sum_{\ell=0}^N\left(\sum_{j=0}^N\omega_j\tilde{\Phi}_{\ell}(x_j)\cdot1\right)\dfrac{\tilde{\Phi}_{\ell}(x)}{1+\lambda\mu_{\ell}^2}\\
&=\|\tilde{\Phi}_0(x)\|_{L_2}\dfrac{\tilde{\Phi}_{0}(x)}{1+\lambda\mu_{0}^2}=\dfrac{1}{1+\lambda\mu_{0}^2}.
\end{split}
\end{equation*}
Note that
\begin{equation*}
1\neq\sum\limits_{j=0}^N\omega_j\sum_{\ell=0}^N\dfrac{\tilde{\Phi}_{\ell}(x_j)\tilde{\Phi}_{\ell}(x)}{1+\lambda\mu_{\ell}^2}\quad\text{when}\quad\lambda\mu_{\ell}^2\neq0,
\end{equation*}
due to the existence of regularization. Then the $\ell_2-$regularized approximation polynomial \eqref{minimizer1} can be expressed as
\begin{equation}\label{l2frac}
p_{L,N+1}^{\ell_2}(x)=\dfrac{\sum\limits_{j=0}^N\left(\omega_j\sum\limits_{\ell=0}^N\dfrac{\tilde{\Phi}_{\ell}(x_j)\tilde{\Phi}_{\ell}(x)}{1+\lambda\mu_{\ell}^2}\right)f(x_j)}{(1+\lambda\mu_0^2)\sum\limits_{j=0}^N\omega_j\sum\limits_{\ell=0}^N\dfrac{\tilde{\Phi}_{\ell}(x_j)\tilde{\Phi}_{\ell}(x)}{1+\lambda\mu_{\ell}^2}}.
\end{equation}
Without loss of generality, assume $\mu_{\ell}=1$ for $\ell\geq N+1$. Note that under this assumption, $\left\{\frac{\tilde{\Phi}_{\ell}(x)}{\sqrt{1+\lambda\mu_{\ell}^2}}\right\}_{\ell\in\mathbb{N}}$ is still a sequence of orthogonal polynomials. By the Christoffel-Darboux formula \citep[Section 1.3.3]{gautschi2004orthogonal}, we rewrite $\sum_{\ell=0}^N\frac{\tilde{\Phi}_{\ell}(x_j)\tilde{\Phi}_{\ell}(x)}{1+\lambda\mu_{\ell}^2}$ in the form
\begin{align}
\sum\limits_{\ell=0}^N\dfrac{\tilde{\Phi}_{\ell}(x)\tilde{\Phi}_{\ell}(x_j)}
{1+\lambda\mu_{\ell}^2}&=\nonumber\dfrac{k_N}{h_Nk_{N+1}}\dfrac{\dfrac{\tilde{\Phi}_{N+1}(x)}
{\sqrt{1+\lambda\mu_{N+1}^2}}\dfrac{\tilde{\Phi}_{N}(x_j)}{\sqrt{1+\lambda\mu_{N}^2}}-\dfrac{\tilde{\Phi}_{N+1}(x_j)}
{\sqrt{1+\lambda\mu_{N+1}^2}}\dfrac{\tilde{\Phi}_{N}(x)}{\sqrt{1+\lambda\mu_{N}^2}}}{x-x_j}\\
&=\dfrac{k_N}{h_Nk_{N+1}}\dfrac{\tilde{\Phi}_{N+1}(x)\tilde{\Phi}_{N}(x_j)}{\sqrt{1+\lambda\mu_{N+1}^2}\sqrt{1+\lambda\mu_{N}^2}(x-x_j)}\label{CDformual},
\end{align}
where $k_{\ell}$ and $h_{\ell}$ denote the leading coefficient and the $L_2$ norm of $\frac{\tilde{\Phi}_{\ell}(x)}{\sqrt{1+\lambda\mu_{\ell}^2}}$, respectively.
Let us combine \eqref{CDformual} with the other expression \eqref{l2frac} of the $\ell_2-$regularized approximation polynomial and cancel the factor $ \dfrac{k_N}{h_Nk_{N+1}}\dfrac{\tilde{\Phi}_{N+1}(x)}{\sqrt{1+\lambda\mu_{N+1}^2}\sqrt{1+\lambda\mu_{N}^2}}$ from both the numerator and the denominator. Together with \eqref{l2frac} and \eqref{CDformual} we obtain the solution to the $\ell_2-$regularized approximation problem in a barycentric form, and we name it the \textit{$\ell_2-$regularized barycentric interpolation formula}:
\begin{equation}\label{L2regubary}
p_{N}^{\ell_2-\text{bary}}(x)=\dfrac{\sum\limits_{j=0}^N\dfrac{\Omega_j}{x-x_j}f(x_j)}{(1+\lambda\mu_{0}^2)\sum\limits_{j=0}^N\dfrac{\Omega_j}{x-x_j}},
\end{equation}
where $\Omega_j=\omega_j\tilde{\Phi}_{N}(x_j)$ is the corresponding barycentric weight at $x_j$. This relation between barycentric weights and Gauss quadrature weights is revealed by \citet*{wang2014explicit}; however, this relation does not lead to the fast computation since it still requires evaluating orthogonal polynomials on $\mathcal{X}_{N+1}$. From the relation they also find the explicit barycentric weights for all classical orthogonal polynomials.

Secondly, we induce the \textit{$\ell_1-$regularized barycentric interpolation formula}.
The $\ell_1-$regularized approximation polynomial \eqref{p:l1} can be expressed as the sum of two terms:
\begin{align}\label{l1difference}
p_{L,N+1}^{\ell_{1}}(x)&=\sum\limits_{\ell=0}^N\frac{S_{\lambda\mu_{\ell}}\left(2\sum\limits_{j=0}^N\omega_j\tilde{\Phi}_{\ell}(x_j)f(x_j)\right)}{2}\tilde{\Phi}_{\ell}(x) \notag \\
&=\sum\limits_{\ell=0}^N\left(\sum_{j=0}^N\omega_j\tilde{\Phi}_{\ell}(x_j)f(x_j)\right)\tilde{\Phi}_{\ell}(x)+\sum_{\ell=0}^Nc_{\ell}\tilde{\Phi}_{\ell}(x),
\end{align}
where
\begin{equation*}
c_{\ell}=\frac{S_{\lambda\mu_{\ell}}(2\alpha_{\ell})}{2}-\alpha_{\ell},\quad \ell=0,1,\ldots,N.
\end{equation*}
The first term in \eqref{l1difference} can be directly written in the barycentric form by letting $\lambda=0$ using the $\ell_2-$regularized barycentric formula. Then let the basis $\{\tilde{\Phi}_{\ell}\}_{\ell=0}^N$ transform into Lagrange polynomials $\{\ell_j(x)\}_{j=0}^N$. By the basis-transformation relation between orthogonal polynomials and Lagrange polynomials \citep{gander2005change}, the second term in \eqref{l1difference} can be represented by Lagrange polynomials in the form
\begin{equation}\label{l1lagrange}
\sum_{\ell=0}^Nc_{\ell}\tilde{\Phi}_{\ell}(x)=\sum_{j=0}^N\left(\sum_{\ell=0}^Nc_{\ell}\tilde{\Phi}_{\ell}(x_j)\right)\ell_j(x).
\end{equation}
With the same procedure of obtaining the barycentric formula from the classical Lagrange interpolation formula in \citet{berrut2004barycentric}, the second term \eqref{l1lagrange} in \eqref{l1difference} can be rewritten as
\begin{equation}\label{additionalterm}
\sum_{\ell=0}^Nc_{\ell}\tilde{\Phi}_{\ell}(x)=\dfrac{\sum\limits_{j=0}^N\dfrac{\Omega_j}{x-x_j}\left(\sum\limits_{\ell=0}^Nc_{\ell}\tilde{\Phi}_{\ell}(x_j)\right)}{\sum\limits_{j=0}^N\dfrac{\Omega_j}{x-x_j}}.
\end{equation}
Together with \eqref{l1difference} and \eqref{additionalterm}, we obtain the \textit{$\ell_1-$regularized barycentric interpolation formula}:
\begin{equation}\label{L1regubary}
p_{N}^{\ell_1-\text{bary}}(x)=\dfrac{\sum\limits_{j=0}^N\dfrac{\Omega_j}{x-x_j}\left(f(x_j)+\sum\limits_{\ell=0}^Nc_{\ell}\tilde{\Phi}_{\ell}(x_j)\right)}{\sum\limits_{j=0}^N\dfrac{\Omega_j}{x-x_j}}.
\end{equation}

If $\lambda=0$, the basis is normalized Legendre polynomials (i.e., interpolation nodes are Legendre points) and $\Omega_j=(-1)^j\sqrt{(1-x_j^2)\omega_j}$ where $\omega_j$ is the Gauss quadrature weight at $x_j$ \citep{wang2012convergence}, then both the $\ell_2-$regularized barycentric interpolation formula \eqref{L2regubary} and the $\ell_1-$regularized barycentric interpolation formula \eqref{L1regubary} reduce to the Wang-Xiang formula \citep{wang2012convergence,trefethen2013approximation}. Inspired by the work of \citet{higham2004numerical}, we will conduct numerical studies on both regularized barycentric interpolation formulae \eqref{L2regubary} and \eqref{L1regubary}, such as numerical stability, see the next paper \citep{an2019the}.

\section{Quality of $\ell_2-$regularized weighted least squares approximation}\label{l2quality}
In this section, we study the quality of the $\ell_2-$regularized weighted least squares approximation in terms of Lebesgue constants. As is known to all, the Lebesgue constant is a tool for quantifying the divergence or convergence of polynomial approximation. From 1910, a lot of works have been done on Lebesgue constants \citep{fejer1910lebesguessche,von1939orthogonal,rivlin1980introduction,powell1981approximation,wang2012convergence,trefethen2013approximation}. This paper considers Lebesgue constants in the case of regularization.
The Lebesgue constant is the $\infty-$norm of the linear mapping from data to approximation polynomial:
\begin{equation*}
{\rm \Lambda}_L:=\underset{f\neq 0}\sup\frac{\|p_{L}\|_{\infty}}{\|f\|_{\infty}}.
\end{equation*}
In particular, we have the following estimation on the $\Lambda_L$ as a consequence of \eqref{p:l2}.
\begin{proposition}
Adopt the notation and assumptions of Theorem \ref{thm:l2}. We have
\begin{align}
{\rm \Lambda}_L=&\underset{x\in[-1,\,1]}\max \sum_{j=1}^{N}w_j\Big| \sum_{\ell=0}^{L}\frac{1}{1+\lambda\mu_{\ell}^2}\tilde{\Phi}_{\ell}(x)\tilde{\Phi}_{\ell}(x_j)\Big|
\nonumber \\  \leq &\underset{x\in[-1,\,1]}\max \sum_{j=0}^{N}w_j \sum_{\ell=0}^{L}\frac{1}{1+\lambda\mu_{\ell}^2}\Big|\tilde{\Phi}_{\ell}(x)\tilde{\Phi}_{\ell}(x_j)\Big|.
\end{align}
\end{proposition}

\subsection{Lebesgue constants with the basis of Chebyshev polynomials of the first kind}
We mimic the discussion of the least squares approximation without regularization in
\citet[Section 2.4]{rivlin1980introduction}. We shall treat the case of normalized Chebyshev polynomials of the first kind $T_{\ell}(x),\,\ell=0,1,\ldots,$ as the basis for $\mathbb{P}_L$. Primary results are also available in \citet{powell1967maximum}. Consider a weighted Fourier series of a given continuous function $g(\theta)$ over $[-\pi,\pi]$:
\begin{equation*}
q_L(\theta)=\frac{\rho_{0,L}}{2}a_0+\sum_{\ell=1}^L\rho_{\ell,L}(a_{\ell}\cos\ell\theta+b_{\ell}\sin\ell\theta),
\end{equation*}
where $a_0,a_1,\ldots,a_L,b_1,\ldots,b_L$ are Fourier coefficients defined as
\begin{equation*}
a_{\ell}=\frac{1}{\pi}\int_{-\pi}^{\pi}g(t)\cos\ell tdt,\quad \ell=0,1,\ldots,L,
\end{equation*}
and
\begin{equation*}
b_{\ell}=\frac{1}{\pi}\int_{-\pi}^{\pi}g(t)\sin\ell tdt,\quad \ell=1,\ldots,L,
\end{equation*}
and weights $\rho_{\ell,L}=1/(1+\lambda\mu_{\ell}^2)$, $\ell=0,1,\ldots,L$.

\begin{lemma}[\citealp{rivlin1980introduction}]\label{qtheta}
If $g(\theta)$ is continuous on $[-\pi,\pi]$ with period $2\pi$, then
\begin{equation*}
q_L(\theta)=\frac{1}{\pi}\int_{-\pi}^{\pi}g(t+\theta)u_L(t)dt,
\end{equation*}
where
\begin{equation*}
u_L(t)=\frac{\rho_{0,L}}{2}+\sum_{\ell=1}^L\rho_{\ell,L}\cos\ell t.
\end{equation*}
\end{lemma}
\begin{definition}
\textit{Lebesgue constants} for $\ell_2-$regularized least squares approximation using Chebyshev polynomials of the first kind are defined as
\begin{equation*}
{\rm \Lambda}_L:=\frac{1}{\pi}\int_{-\pi}^{\pi}|u_L(t)|dt.
\end{equation*}
\end{definition}
The case of $\lambda=0$ leads to Lebesgue constants for Fourier series (without regularization) \citep[Section 2.4]{rivlin1980introduction} in the form of
\begin{equation*}
{\rm \Lambda}_L=\frac{1}{\pi}\int_{0}^{\pi}\frac{|\sin(L+\frac12)t|}{\sin\frac{t}{2}}dt=\frac{1}{2\pi}\int_{-\pi}^{\pi}\Big|\frac{\sin(L+\frac12)t}{\sin\frac{t}{2}}\Big|dt,
\end{equation*}
where the last integrand is the famous Dirichlet kernel. For estimation of $\Lambda_L$, we have the following lemma.

\begin{lemma}\label{dirichletkernel}
Let $D_n$ denote the Dirichlet kernel \citep{stein2003princeton}
\begin{equation*}
D_n(x):=\sum_{k=-n}^ne^{ikx}=\frac{\sin(n+\frac12)x}{\sin\frac x2}.
\end{equation*}
Then for~$n\geq2$,
\begin{equation}\label{eq:LebsgueconstantCheby}
\frac{1}{2\pi}\int_{-\pi}^{\pi}|D_n(x)|dx=\frac{4}{\pi^2}\log n+\mathcal{O}(1),
\end{equation}and
\begin{equation}\label{ineq:eta1}
\frac{1}{2\pi}\int_{-\pi}^{\pi}|D_n(x)|dx\leq\frac{4}{\pi^2}\log(n-1)+\eta,
\end{equation}
where $\eta=\frac{4}{\pi^2}+\frac{2}{\pi}(1+\int_{0}^{\pi}\frac{\sin x}{x}dx)=2.220884\ldots$.
\end{lemma}
\eqref{eq:LebsgueconstantCheby} is a known result given by \citet{fejer1910lebesguessche}, one might find in \citet[p.5]{lorentz1966approximation} and \citet[Section 2.7.2]{stein2003princeton}. \eqref{ineq:eta1} is sharper than the known result ${\rm\Lambda}_L<\frac{4}{\pi^2}\log n+3,\, n\geq2$ \cite[Lemma 2.2]{rivlin1980introduction}.
 For completeness, the proof of Lemma \ref{dirichletkernel} is given in Appendix: Proof of Lemma \ref{dirichletkernel}.

\begin{theorem}\label{prop:l2}
Suppose $f$ is continuous on $[-1,1]$, and normalized Chebyshev polynomials constitute a basis for $\mathbb{P}_L$. Then Lebesgue constants ${\rm\Lambda}_L$ for the $\ell_2-$regularized least squares approximation of degree $L$ ($L\geq2$) on $[-1,1]$ satisfy
\begin{equation}\label{lebesguebounds}{\rm \Lambda}_L=\frac{4\log L/\pi^2}{1+\lambda\tilde{\mu}^2}+\mathcal{O}(1),
\end{equation}\text{and}
\begin{equation}\label{ineq:lebesguebounds}
{\rm\Lambda}_L\leq\frac{4\log (L-1)/\pi^2+\eta}{1+\lambda\underline{\mu}^2},\end{equation}
where $\underline{\mu}=\min\{\mu_0,\mu_1,\ldots,\mu_L\}$, $\tilde{\mu}\in\mathbb{R}$ which satisfies $\min\{\mu_0,\mu_1,\ldots,\mu_L\}\leq\tilde{\mu}\leq\max\{\mu_0,\mu_1,\ldots,\mu_L\}$ and $\eta=2.220884\ldots$.
\end{theorem}
\begin{proof}
Since $f$ is continuous on $[-1,1]$, then $g(\theta)=f(\cos\theta)$ is continuous on $[0,\pi]$. If $g(-\theta)=g(\theta)$, then $g$ is continuous on $[-\pi,\pi]$. The even function $g$ gives $b_{\ell}=0$ for all $\ell=1,\ldots,L$, and then
\begin{align}
q_L(\theta)&=\nonumber\frac{\rho_{0,L}}{2}a_0+\sum_{\ell=1}^L\rho_{\ell,L}a_{\ell}\cos\ell\theta\\
&=\frac{\rho_{0,L}}{2}a_0+\sum_{\ell=1}^L\rho_{\ell,L}a_{\ell}T_{\ell}(x)\\ \nonumber
&=\frac{\sqrt{\pi/2}\rho_{0,L}}{2}a_0\tilde{T}_0+\sum_{\ell=1}^L\sqrt{\pi}\rho_{\ell,L}a_{\ell}\tilde{T}_{\ell}(x),
\end{align}
which reveals that this is the $\ell_2-$regularized least squares approximation of degree $L$ with the basis for $\mathbb{P}_L$ being normalized Chebyshev polynomials of the first kind. Since $g(\theta)$ is continuous on $[-\pi,\pi]$ with period $2\pi$, there must exist $M\geq0$ such that
\begin{equation*}
|g(t+\theta)|\leq M,\quad t,\theta\in[-\pi,\pi].
\end{equation*}
By Lemma \ref{qtheta},
\begin{equation*}
\max\limits_{\theta\in[-\pi,\pi]}|q_L(\theta)|\leq M{\rm \Lambda}_L.
\end{equation*}
When $\lambda=0$, one may easily verify that
\begin{equation*}
u_L(t)=\frac12+\sum_{\ell=0}^L\cos\ell t=\frac{\sin(L+\frac12)t}{2\sin\frac t2},
\end{equation*}
then by Lemma \ref{dirichletkernel},
\begin{equation*}
{\rm \Lambda}_L=\frac{1}{2\pi}\int_{-\pi}^{\pi}\Big|\frac{\sin(L+\frac12)t}{\sin\frac{t}{2}}\Big|dt\leq\frac{4}{\pi^2}\log (L-1)+\eta,
\end{equation*}
and
\begin{equation*}
{\rm \Lambda}_L=\frac{4}{\pi^2}\log L+\mathcal{O}(1).
\end{equation*}
Let $\underline{\mu}=\min\{\mu_0,\mu_1,\ldots,\mu_L\}$ and $\overline{\mu}=\max\{\mu_0,\mu_1,\ldots,\mu_L\}$. And let $\tilde{\mu}$ be a real number satisfying $\underline{\mu}\leq\tilde{\mu}\leq\overline{\mu}$. If $\lambda>0$, then
\begin{equation*}
u_L(t)\leq\frac{1}{1+\lambda\underline{\mu}^2}\left(\frac12+\sum_{\ell=0}^L\cos\ell t\right)=\frac{1}{1+\lambda\underline{\mu}^2}\frac{\sin(L+\frac12)t}{2\sin\frac t2},
\end{equation*}
and
\begin{equation*}
u_L(t)=\frac{1}{1+\lambda\tilde{\mu}^2}\left(\frac12+\sum_{\ell=0}^L\cos\ell t\right)=\frac{1}{1+\lambda\tilde{\mu}^2}\frac{\sin(L+\frac12)t}{2\sin\frac t2},
\end{equation*}
which gives the asymptotic result \eqref{lebesguebounds} and the inequality \eqref{ineq:lebesguebounds}.
\end{proof}
%\begin{remark}
%From the proof we know that the case of $\lambda=0$ is reduced into the bounds of Lebesgue constants for the Chebyshev truncation given in \citet[Section 2.4]{rivlin1980introduction}.
%\end{remark}

Take the family of normalized Chebyshev polynomials of the first kind $\{\tilde{T}_{\ell}(x)\}_{\ell=0}^L$ as the basis for $\mathbb{P}_L$ and the set of zeros of $\tilde{T}_{L+1}(x)$ as the set of nodes. Setting  $L=N$ and $\lambda=10^{-1}$. Fig. \ref{fig:lebesgue} illustrates Lebesgue constants with respect to different choices of regularization parameter $\lambda$.
\begin{figure}[htbp]
  \centering
  \includegraphics[width=10cm]{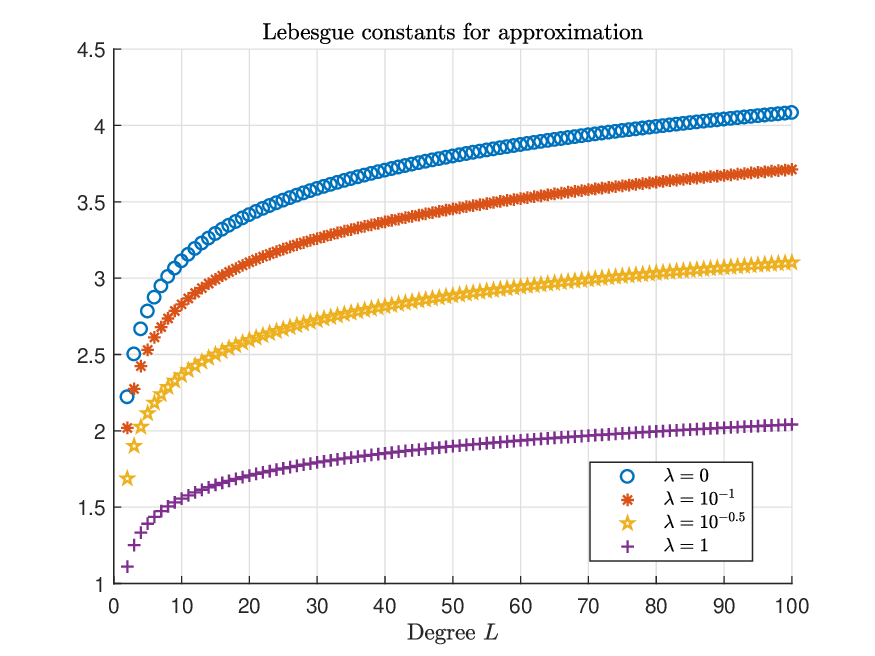}\\
  \caption{The Lebesgue constant of $\ell_2-$regularized approximation with $L=N$ and $\mu_{\ell}=1$ for $\ell=0,1,\ldots,L$ using Chebyshev polynomials of the first kind}\label{fig:lebesgue}
\end{figure}

\subsection{Lebesgue constants with the basis of Legendre polynomials}
In this subsection, we will derive asymptotic bounds of Lebesgue constants of the $\ell_2-$ regularized approximation by using Legendre polynomials. Consider the kernel
\begin{equation*}
K_L(x,y):=\sum_{\ell=0}^L\frac{2\ell+1}{1+\lambda\mu_{\ell}^2}P_{\ell}(x)P_{\ell}(y),
\end{equation*}
where $P_{\ell}(\cdot)$ denotes the Legendre polynomial of degree $\ell$. The case of $\lambda=0$ gives a simple kernel
\begin{equation*}
T_L(x,y)=\sum_{\ell=0}^L(2\ell+1)P_{\ell}(x)P_{\ell}(y)=(L+1)\frac{P_L(x)P_{L+1}(y)-P_{L+1}(x)P_L(y)}{y-x},
\end{equation*}where the last equality is due to \citet[(4)]{gronwall1913laplacesche}.
Thus,
\begin{equation*}
K_L(x)\triangleq K_L(x,1):=\sum_{\ell=0}^L\frac{2\ell+1}{1+\lambda\mu_{\ell}^2}P_{\ell}(x),
\end{equation*}
and
\begin{equation*}
T_L(x)\triangleq T_L(x,1):=\sum_{\ell=0}^L(2\ell+1)P_{\ell}(x)=(L+1)\frac{P_L(x)-P_{L+1}(x)}{1-x},
\end{equation*}
where the rightmost equality is due to \citet[(5)]{gronwall1913laplacesche}.
 %or the Christoffel-Darboux formula \citep[Section 1.3.3]{gautschi2004orthogonal}.
The reader may note that
\begin{equation}\label{K<S}
|K_L(x)|\leq\frac{1}{1+\lambda\underline{\mu}^2}|T_L(x)|,
\end{equation}
where $\underline{\mu}=\min\{\mu_0,\mu_1,\ldots,\mu_L\}$.
\begin{definition}
\textit{Lebesgue constants} for the $\ell_2-$regularized approximation using Legendre polynomials are defined as
\begin{equation*}
{\rm \Lambda}_L:=\frac12\int_{-1}^1|K_L(x)|dx.
\end{equation*}
\end{definition}
The case of $\lambda=0$ leads to
\begin{equation}\label{eq:theta}
{\rm \Theta}_L:=\frac12\int_{-1}^1|T_L(x)|dx=\frac{L+1}{2}\int_{-1}^1\left|\frac{P_L(x)-P_{L+1}(x)}{1-x}\right|dx,
\end{equation}
which is the definition of Lebesgue constant of Legendre truncation of degree $L$ \citep{gronwall1913laplacesche}.
\begin{lemma}[\citealp{gronwall1913laplacesche} and \citealp{von1934einige}]\label{gronwall}
Let ${\rm \Theta}_L$ be defined as \eqref{eq:theta}.
Then
\begin{equation}
\lim\limits_{L\rightarrow\infty}\frac{{\rm \Theta}_L}{\sqrt{L}}=2\sqrt{\frac{2}{\pi}}.
\end{equation}
\end{lemma}
\noindent Combining Lemma \eqref{gronwall} with the inequality \eqref{K<S}, we obtain the estimation of ${\rm \Lambda}_L$ in the case of Legendre polynomials.
\begin{theorem}\label{legendrelebesgue}
Suppose $f$ is continuous on $[-1,1]$, and Legendre polynomials constitute the basis for $\mathbb{P}_L$. Then Lebesgue constants ${\rm\Lambda}_L$ for the $\ell_2-$regularized least squares approximation of degree $L$ ($L\geq2$) on $[-1,1]$ satisfy
\begin{equation*}
{\rm\Lambda}_L\leq\frac{1}{1+\lambda\underline{\mu}^2}\left(\frac{2^{3/2}}{\sqrt{\pi}}L^{1/2}+o(L^{1/2})\right),
\end{equation*}
where $\underline{\mu}=\min\{\mu_0,\mu_1,\ldots,\mu_L\}$.
\end{theorem}
The proof for Theorem \ref{legendrelebesgue} is based on the above discussion.

\section{Sparsity of solution to the $\ell_1-$regularized approximation problem}\label{l1sparsity}
Some real-world problems, such as signal processing, often have sparse solutions. One may seek the sparsest solution to a problem, that is, the solution containing zero elements at most. However, a vector of real data would rarely contains many strict zeros. One may introduce other methods to seek sparsity, such as $\min_x$ $\|x\|_p$, where $\|x\|_p=(\sum_i|x_i|^p)^{1/p}$, $0<p<1$. Nevertheless, these optimization problems mentioned above are nonconvex and nondifferentiable \citep{clarke1990optimization,bruckstein2009sparse}. Regularized methods, especially $\ell_1-$regularizaed methods, also produce sparse solutions, according to our examples. One may find a relatively sparse solution by minimizing $\ell_1$ norm, because such an optimization problem is a convex optimization problem, and its solution is the closest one to the sparsest solution for all $p\geq1$ in $\|x\|_p$. For topics on sparsity, we refer to \citet{bruckstein2009sparse}. We consider the sparsity of the solution ${\bm \beta}$ to the $\ell_1-$regularized approximation problem \eqref{problem:l1regumatrix}. The sparsity is measured by the number of nonzero elements of $\bm\beta$, denoted as $\|{\bm\beta}\|_0$, also known as the zero ``norm" (it is not a norm actually) of $\bm\beta$.

Before discussing the upper bound of $\|{\bm\beta}\|_0$, we offer a quick glimpse into the zero elements distribution of the $\ell_1-$regularized approximation solution. From Definition \ref{equ:soft} of the soft threshold operator, we have the following result.

\begin{proposition}[Zero elements distribution of the $\ell_1-$regularized approximation solution]
Adopt the notation and assumptions of Lemma \ref{thm:diagonal}. If $\mu_{\ell}$ satisfies
\begin{equation}
-\lambda\mu_{\ell}\leq2\sum\limits_{j=0}^N\omega_j\tilde{\Phi}_{\ell}(x_j)f(x_j)\leq\lambda\mu_{\ell},
\end{equation}
then its corresponding $\beta_{\ell}$ is zero, $\ell=0,1,\ldots,L$.
\end{proposition}

If $\lambda>0$, then $\|{\rm\bf{A}}_L^T{\bf{W}}{\rm\bf{f}}\|_0$ becomes an upper bound of the number of nonzero elements of ${\bm{\beta}}$. Furthermore, we obtain the exact number of nonzero elements of ${\bm{\beta}}$ with the help of the information of ${\bm{\beta}}$.
\begin{theorem}\label{sparsityyesregu}
Let ${\bm{\beta}}=[\beta_0,\beta_1,\ldots,\beta_L]^T$ be the solution to the $\ell_1-$regularized problem \eqref{problem:l1regumatrix}. If $\lambda>0$, then the number of nonzero elements of ${\bm{\beta}}$ satisfies
\begin{equation}\label{yesregu}
\|\bm{\beta}\|_0\leq\|{\rm\bf{A}}_L^T{\bf{W}}{\rm\bf{f}}\|_0,
\end{equation}
and
\begin{equation}\label{exactregu}
\|\bm{\beta}\|_0=\|{\rm\bf{A}}_L^T{\bf{W}}{\rm{\bf{f}}}\|_0-\#\text{\{occurrences of $\beta_{\ell}=0$ but $\alpha_{\ell}\neq0$\}},
\end{equation}
where ${\rm \#}\text{\{occurrences of $\beta_{\ell}=0$ but $\alpha_{\ell}\neq0$\}}$ denotes the number of occurrences of $\beta_{\ell}=0$ but $\alpha_{\ell}\neq0$ for $\ell=0,1,\ldots,L$.
\end{theorem}
\begin{proof}
The first order condition \eqref{l1deri} of the $\ell_1-$regularized approximation problem \eqref{problem:l1regumatrix} in the matrix form can be rewritten as
\begin{equation*}
{\bm \beta}\in{\rm \bf{A}}_L^T{\bf{W}}{\rm \bf{f}}-\frac{\lambda\partial(\|{\rm \bf{R}}_L\bm{\beta}\|_1)}{2}.
\end{equation*}
To obtain a solution, there must exist an $L+1$ vector ${\rm\bf{h}}=[h_0,h_1,\ldots,h_L]^T\in\partial(\|{\rm \bf{R}}_L\bm{\beta}\|_1)$ such that
\begin{equation}\label{betah}
{\bm \beta}={\rm \bf{A}}_L^T{\bf{W}}{\rm \bf{f}}-\frac{\lambda{\rm\bf{h}}}{2}.
\end{equation}
Since $\mu_{\ell}>0$ for all $\ell=0,1,\ldots,L$, elements of ${\rm\bf{h}}$ satisfy
\begin{equation*}
h_{\ell}=\begin{cases}
\mu_{\ell},& \mu_{\ell}\beta_{\ell}>0,\text{~i.e.,~}\beta_{\ell}>0\\
-\mu_{\ell},& \mu_{\ell}\beta_{\ell}<0,\text{~i.e.,~}\beta_{\ell}<0\\
r\mu_{\ell}~\forall r\in[-1,1],& \mu_{\ell}\beta_{\ell}=0,\text{~i.e.,~}\beta_{\ell}=0,
\end{cases}
\end{equation*}
yielding $\|{\rm{\bf{h}}}\|_0\geq\|{\bm{\beta}}\|_0$. Expression \eqref{betah} gives
\begin{equation*}
\left\|\bm{\beta}+\frac{\lambda{\rm{\bf{h}}}}{2}\right\|_0=\left\|{\rm \bf{A}}_L^T{\bf{W}}{\rm \bf{f}}\right\|_0.
\end{equation*}
If $\beta_{\ell}>(\text{or}<)0$, then $h_{\ell}>(\text{or}<)0$. If $\beta_{\ell}=0$, whereas $h_{\ell}$ may not be zero. Thus,
\begin{equation*}
\left\|\bm{\beta}+\frac{\lambda{\rm{\bf{h}}}}{2}\right\|_0=\|{\rm{\bf{h}}}\|_0.
\end{equation*}
Hence,
\begin{equation*}
\|{\bm{\beta}}\|_0\leq\|{\rm{\bf{h}}}\|_0=\|{\rm \bf{A}}_L^T{\bf{W}}{\rm \bf{f}}\|_0.
\end{equation*}
Let $\beta_{\ell}^*$ denote the optimal solution to the problem \eqref{equ:optimal}
. With the aid of the closed-form solution to the $\ell_1-$regularized approximation problem (see \eqref{beta:l1}), expression \eqref{betah} gives birth to
\begin{equation*}
\frac{h_{\ell}}{\mu_{\ell}}=\frac{2}{\lambda\mu_{\ell}}(\alpha_{\ell}-\beta_{\ell}^*)=\begin{cases}
1,&\beta_{\ell}^*>0,\\
-1,&\beta_{\ell}^*<0,\\
\frac{2\alpha_{\ell}}{\lambda\mu_{\ell}},&\beta_{\ell}^*=0.
\end{cases}
\end{equation*}
Due to $\|\rm{\bf{h}}\|_0=\left\|\frac{\rm{\bf{h}}}{\bm{\mu}}\right\|_0$, where $\frac{\rm{\bf{h}}}{\bm{\mu}}$ denotes the pointwise division between $\rm{\bf{h}}$ and $\bm{\mu}$, the difference between $\|\rm{\bf{h}}\|_0$ and $\|\bm{\beta}\|_0$ is expressed as
\begin{equation}\label{h-beta}
\|\rm{\bf{h}}\|_0-\|\bm{\beta}\|_0=\left\|\frac{\rm{\bf{h}}}{\bm{\mu}}\right\|_0-\|\bm{\beta}\|_0=\#\text{\{occurrences of $\beta_{\ell}=0$ but $\alpha_{\ell}\neq0$\}}.
\end{equation}
Together with $\|\rm{\bf{h}}\|_0=\|{\rm\bf{A}}_L^T{\bf{W}}{\rm\bf{f}}\|_0$ and \eqref{h-beta}, we obtain the exact number \eqref{exactregu} of nonzero elements of ${\bm{\beta}}$.
\end{proof}

\begin{corollary}\label{sparsitynoregu}
If $\lambda=0$, then the number of nonzero elements of ${\bm{\beta}}$ satisfies
\begin{equation}\label{noregu}
\|\bm{\beta}\|_0=\|{\rm\bf{A}}_L^T{\bf{W}}{\rm\bf{f}}\|_0.
\end{equation}
\end{corollary}

\begin{remark}
Together, Theorem \ref{sparsityyesregu} and Corollary \ref{sparsitynoregu} state that the regularized minimization is better than minimization without regularization in terms of sparsity.
\end{remark}

Let the basis for $\mathbb{P}_L$ be the family of normalized Chebyshev polynomials of the first kind $\{\tilde{T}_{\ell}(x)\}_{\ell=0}^L$ and the set of nodes be the set of zeros of $\tilde{T}_{N+1}(x)$. With degree $L$ of approximation polynomial ranging from $1$ to $60$, $\lambda$ evaluated $10^{-1}$ and $\mu_{\ell}$ evaluated $1$ for all $\ell=0,1,\ldots,L$, Fig. \ref{sparsity} gives four examples on the bounds and estimations given above.
\begin{figure}[htbp]
  \centering
  \includegraphics[width=\textwidth]{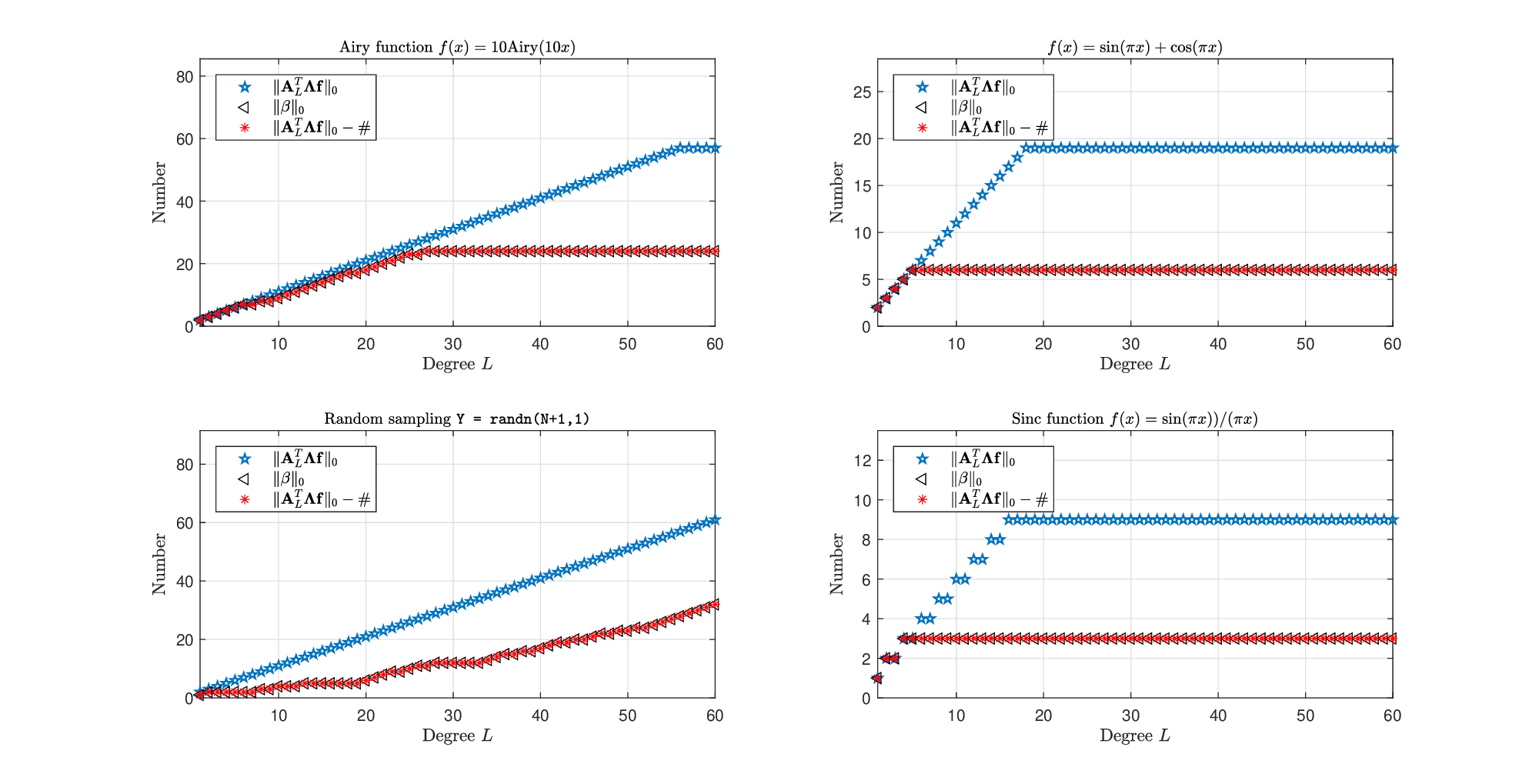}\\
  \caption{Examples on bounding the number of nonzero elements, where \# denotes $\#\text{\{occurrences of $\beta_{\ell}=0$ but $\alpha_{\ell}\neq0$\}}$}\label{sparsity}
\end{figure}
\section{Numerical experiments}\label{numericalexperiments}
In this section, we report numerical results to illustrate the theoretical results derived above and test the efficiency of the $\ell_{2}-$regularized approximation polynomial $p_{L,N+1}^{\ell_2}$ \eqref{p:l2} and the $\ell_{1}-$regularized approximation polynomial $p_{L,N+1}^{\ell_1}$ \eqref{p:l1}. The choices of the basis for $\mathbb{P}_L$ and the set of nodes $\mathcal{X}_{N+1}$ are primary when using both models. We choose Chebyshev polynomials of the first kind and the corresponding Chebyshev points. Certainly, choosing other orthogonal polynomials such as Legendre polynomials is also possible. All computations are performed in MATLAB in double precision arithmetic. Some related commands, for instance, obtaining quadrature points and weights, are included in \textsc{Chebfun} 5.7.0 \citep{trefethon2017chebfun}.

To test the efficiency of approximation, we define the uniform error and the $L_2$ error to measure the approximation error:
\begin{itemize}
  \item The uniform error of the approximation is estimated by
  \begin{equation*}\begin{split}
  \|f(x)-p_{L,N+1}(x)\|_{\infty} &:=\underset{x\in[-1,1]}{\max}|f(x)-p_{L,N+1}(x)|\\
                           &\simeq\underset{x\in\mathcal{X}}{\max}|f(x)-p_{L,N+1}(x)|,
  \end{split}\end{equation*}
  where $\mathcal{X}$ is a large but finite set of well distributed points (for example, clustered grids, see \citet[Chapter 5]{trefethen2000spectral}) over the interval $[-1,1]$.
  \item The $L_2$ error of the approximation is estimated by a proper Gauss quadrature rule:
  \begin{equation*}\label{l2error}\begin{split}
  \|f(x)-p_{L,N+1}(x)\|_{L_2} &=\left(\int_{-1}^1w(x)(f(x)-p_{L,N+1}(x))^2dx\right)^{1/2}\\
                        &\simeq\left(\sum_{j=0}^N\omega_j(f(x_j)-p_{L,N+1}(x_j))^2\right)^{1/2}.
  \end{split}\end{equation*}
\end{itemize}
\subsection{Regularized approximation models for exact data}\label{testexact}
The fact should always stick in readers' mind that regularization is introduced to solve ill-posed problems or to prevent overfitting. When approximation applies to functions without noise, regularization parameter $\lambda=0$ (no regularization) contributes to the best choice of approximating. Fig. \ref{smoothapproximation} reports the efficiency and errors for approximating function
\begin{equation*}
f_1(x)=\tanh(20\sin(12x))+0.02\text{e}^{3x}\sin(300x),
\end{equation*}
with or without regularization over $[-1,1]$. The test function is given in \citet{trefethen2013approximation}. Let $N = 600$, $L = 200$, $\lambda = 10^{-1}$ and $\mu_{\ell} = 1$ for all $\ell = 0,1,...,L$. Fig. \ref{smoothapproximation} illustrates that regularization is beyond use in this well-posed approximation problem, and $\ell_2-$regularization is better than $\ell_1-$regularization in approximating smooth functions.
\begin{figure}[htbp]
  \centering
  \includegraphics[width=\textwidth]{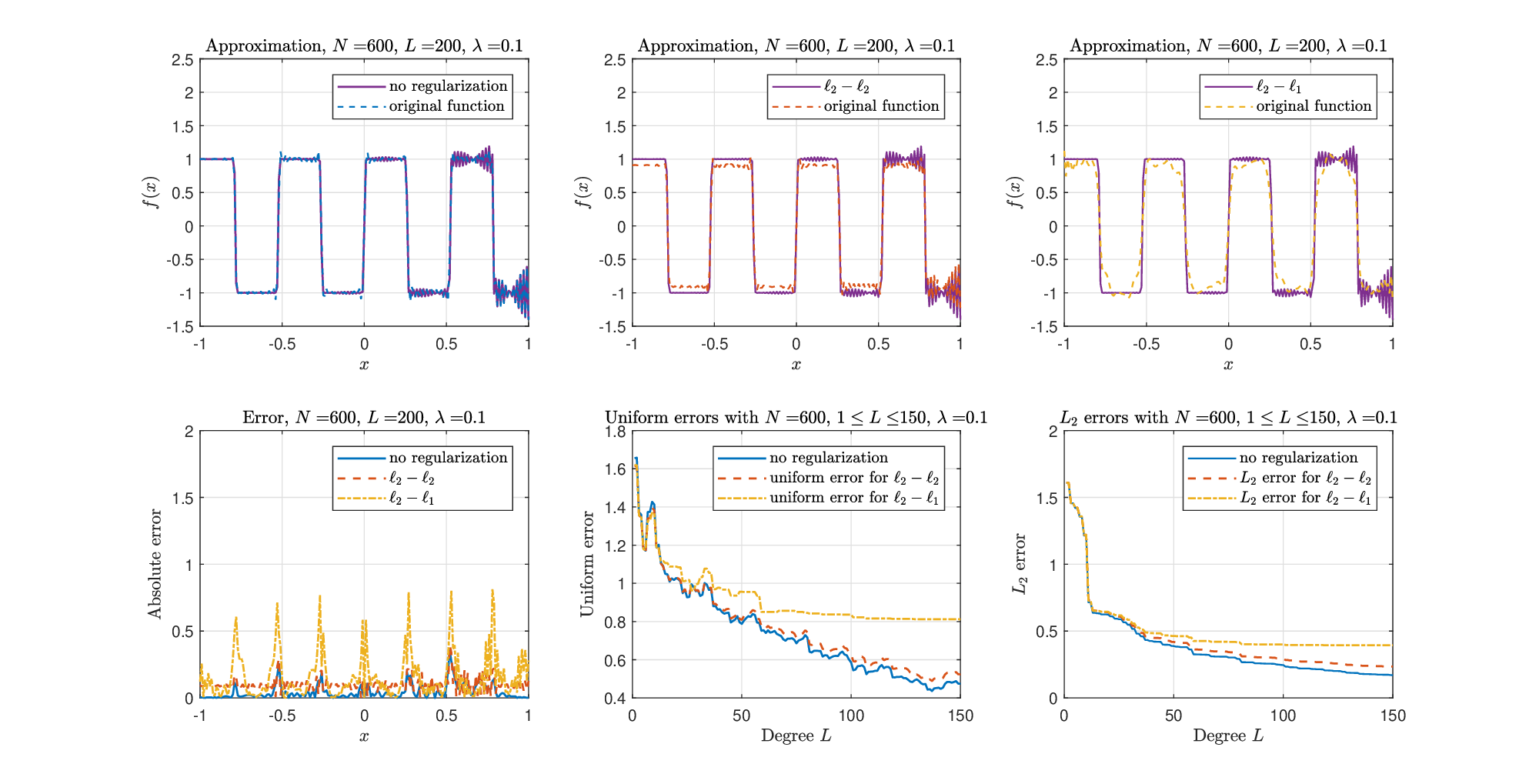}\\
  \caption{Approximation of function $f_1(x)=\tanh(20\sin(12x))+0.02\text{e}^{3x}\sin(300x)$ with exact values on the Chebyshev polynomials of the first kind}\label{smoothapproximation}
\end{figure}

\subsection{Regularized approximation models for contaminated data}\label{generalfunctiondenoising}
We consider
\begin{equation}\label{gatelike}
\displaystyle
f_2(x)=\begin{cases}
\frac{5\sin(5\pi x)}{5 \pi x},&x\neq0,\vspace{0.2cm}\\
5,&x=0,
\end{cases}
\end{equation}
which is the Fourier transform of the gate signal
\begin{equation*}
g(t)=\begin{cases}
1,& |t|\leq5/2,\\
0,& |t|>5/2,
\end{cases}
\end{equation*}
see \citet{bracewell1965the}. We use regularized least squares models to reduce Gaussian white noise added to the function \eqref{gatelike} with the signal-noise ratio (SNR) 10 dB.
The choice of $\lambda$ is critical in these models, so we first consider the relation between $\lambda$ and approximation errors to choose the optimal $\lambda$. Let $L=30$ and $N=100$. We take $\lambda=10^{-15},~10^{-14.5}~10^{-14},~,\ldots,10^{4.5},~10^{5}$ to choose the best regularization parameter. Here we choose $\lambda=10^{-1}$. For more advanced methods to choose the parameter $\lambda$, we refer to \citet{lazarov2007balancing} and \citet{pereverzyev2015parameter} for a further discussion.

Fig. \ref{fig:recover} shows that the $\ell_2-$ and $\ell_1-$regularized approximation models with $\lambda=10^{-1}$ are effective in recovering the noisy function. In the case we let
\begin{equation*}
\mu_{\ell}=\frac{1}{F(\ell/L)},\quad \ell=0,1,\ldots,L,
\end{equation*}
where the filter function $F$ is defined as \citep{an2012regularized}
\begin{equation*}
F(x):=\begin{cases}
1,& x\in[0,1/2],\\
\sin^2\pi x,& x\in[1/2,1],\\
0,& x\in[1,+\infty].
\end{cases}
\end{equation*}
In this case, $\{\mu_{\ell}\}_{\ell=0}^L$ is a sequence of nonnegative nondecreasing parameters.
\begin{figure}[htbp]
  \centering
  \includegraphics[width=\textwidth]{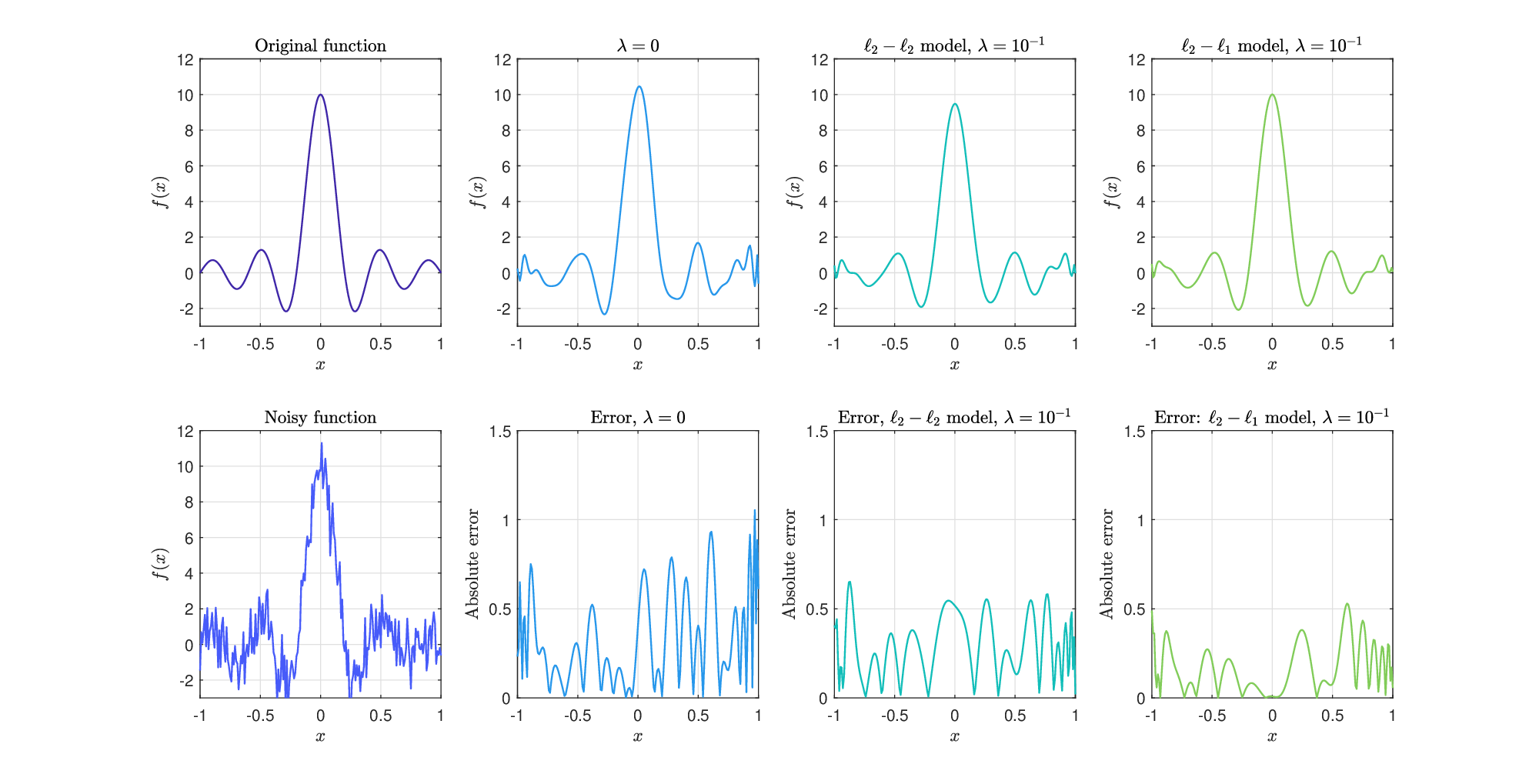}\\
  \caption{Regularized approximation models with $L=30$ and $N=100$ to recover the function \eqref{gatelike} from contaminated data on the Chebyshev polynomials of the first kind}\label{fig:recover}
\end{figure}

Results in Fig. \ref{fig:error} illustrate that the $\ell_1-$regularized approximation model is the best choice when recovering a contaminated function, which accords with the known facts \citep{lu2009sparse}. Both Fig. \ref{fig:recover} and Fig. \ref{fig:error} show that regularized models are better than those without regularization ($\lambda=0$).
\begin{figure}[htbp]
  \centering
  \includegraphics[width=\textwidth]{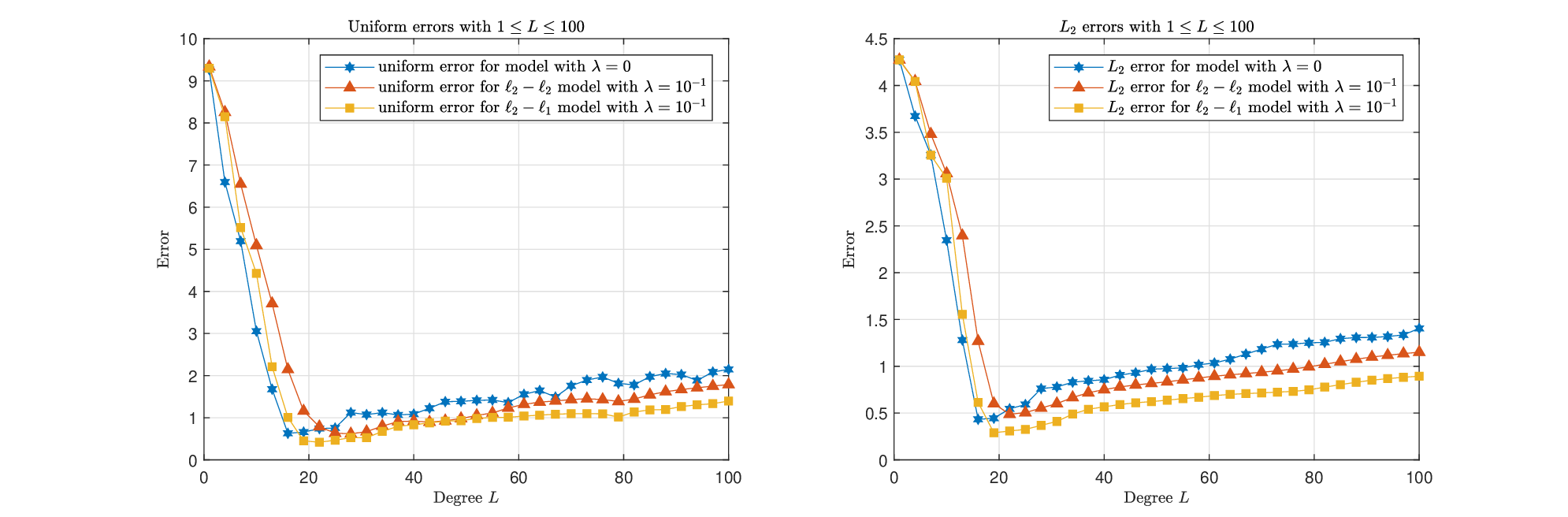}\\
  \caption{Errors for regularized approximation model to recover $f_2(x)$ with fixed $N=100$ on the Chebyshev polynomials of the first kind}\label{fig:error}
\end{figure}

Besides this, consider the highly oscillatory function $$f_3(x)=\text{Airy}(40x),\quad  x\in[-1,1],$$ with 12dB Gauss white noise added (noisy function is shown in Fig. \ref{highoscillated}). We use regularized barycentric formulae \eqref{L2regubary} and \eqref{L1regubary} to conduct this experiment. Let $L=N=500$ and $\{\mu_{\ell}\}_{\ell=0}^L$ be the same as above. Different values of $\lambda$, say $10^{-1}$, $10^{-2}$, $10^{-5}$, lead to different results, see Fig. \ref{highoscillated}.
This experiment indicates that one could apply a simple formula to reduce noise, rather than employ an iterative scheme.
\begin{figure}[htbp]
  \centering
  \includegraphics[width=\textwidth]{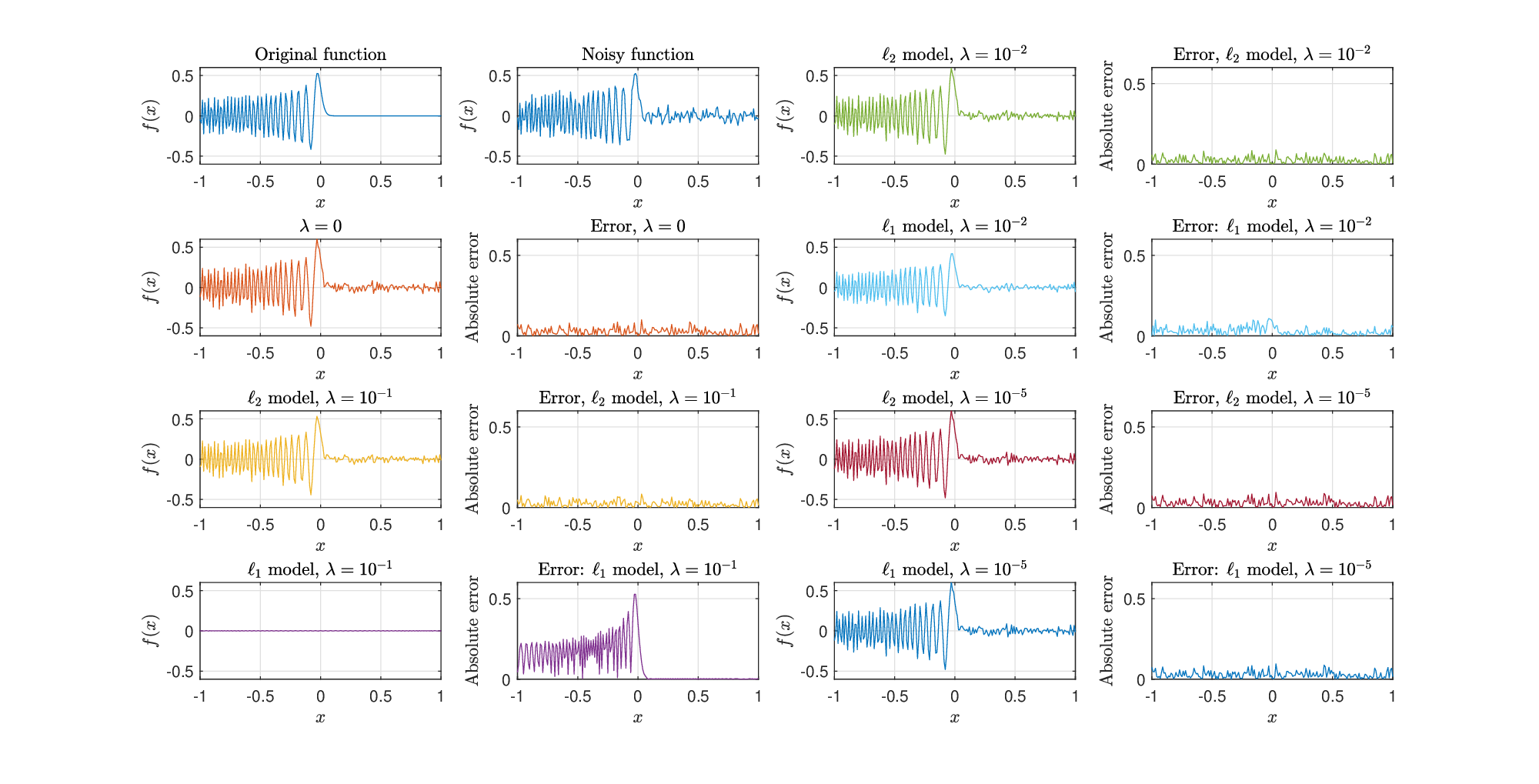}\\
  \caption{Denoising by regularized barycentric formulae with $L=N=500$ on Chebyshev points of the first kind: function $f_3(x)=\text{Airy}(40x)$ with 12dB Gauss white noise added }\label{highoscillated}
\end{figure}

These numerical examples illustrate that for some problems, $\ell_2-$regularization also can be better than $\ell_1-$regularization. For example, $\lambda=10^{-1}$ suits $\ell_2-$regularization, but almost straightens the function with $\ell_1-$regularization. Besides this, we can see that $\ell_2-$regularization contains lower sensitivity than $\ell_1-$regularization.

\section{Concluding remarks}
In this paper, we have investigated minimizers to the $\ell_2-$ and $\ell_1-$regularized least squares approximation problems with the aid of Gauss quadrature points and orthogonal polynomials on the interval $[-1,1]$. Based on those explicit constructed approximation polynomials \eqref{p:l2} and \eqref{p:l1}, the $\ell_2$-regulariarized barycentric interpolation formula \eqref{L2regubary} and the $\ell_1-$regulariarized barycentric interpolation formula \eqref{L1regubary} have been derived. In addition, Lebesgue constants are studied in the case of using Legendre polynomials and normalized Chebyshev polynomials of the first kind as the basis for the polynomial space $\mathbb{P}_L$. A bound for sparsity of the solution to $\ell_1-$regularized approximation is obtained by the refinement of the subgradient. Numerical results indicate that both the $\ell_2-$regularized approximation polynomial and the $\ell_1-$regularized approximation polynomial are practicable and efficient. Regularization parameter choice strategies and error bounds of approximation should be studied in future work.  These results provide new insights into the $\ell_2-$ and $\ell_1-$regularized approximation, and can be adapted to many practical applications such as noise reduction by using the barycentric interpolation scheme on Gauss quadrature points.

%\section*{Acknowledgements}

%We wish to thank Prof. Haiyong Wang (Huazhong University of Science and Technology), Prof. Anthony Man-Cho So (The Chinese University of Hong Kong) and Prof. Xiaoming Yuan (The University of Hong Kong) for their hospitality and interest in this work.

\bibliographystyle{IMANUM-BIB}
\bibliography{IMANUM-refs}

\clearpage

\appendix

\section*{Appendix A.\ Proof of Lemma \ref{dirichletkernel}}
\label{app1}
\begin{proof} We adopt the proof of
\citet[Lemma 2.2]{rivlin1980introduction}, but bring up a sharper bound of this inequality:
\begin{equation}\label{ineq:R1}
\frac{1}{2\pi}\int_{-\pi}^{\pi}|D_n(x)|dx=\frac{1}{2\pi}\int_{-\pi}^{\pi}\Big|\frac{\sin(n+\frac12)x}{\sin\frac x2}\Big|dx<\frac{4}{\pi^2}\log n+3,\quad\quad n\geq1.
\end{equation}
When $n=1$, $\frac{1}{2\pi}\int_{-\pi}^{\pi}|D_n(x)|dx=1$.
 Suppose $n\geq2$, then
%\begin{equation}
\begin{align}
\frac{1}{2\pi}\int_{-\pi}^{\pi}\Big|\frac{\sin(n+\frac12)x}{\sin\frac x2}\Big|dx\nonumber
&=\frac{1}{\pi}\int_{0}^{\pi}\frac{|\sin(n+\frac12)x|}{\sin\frac x2}dx\\ \nonumber
&=\frac{1}{\pi}\int_{0}^{\pi}\left|\frac{\sin nx}{\tan\frac x2}+\cos nx\right|dx\\
&\leq \frac{1}{\pi}\int_0^{\pi}\frac{|\sin nx|}{\tan \frac x2}dx+\frac{1}{\pi}\int_0^{\pi}|\cos nx|dx.\label{expression1}
\end{align}We have
%\end{equation}
%\begin{itemize}
%  \item
\begin{equation}\label{expression2}
\frac{1}{\pi}\int_0^{\pi}|\cos nx|dx=\frac{2}{\pi},
\end{equation}
%  \item
%The fact that $\tan x\geq x$ for $0\leq x\leq \frac{\pi}{2}$ leads to
%\begin{equation*}
%\int_0^{\pi}\frac{|\sin nx|}{\tan\frac x2}dx\leq2\int_0^{\pi}\frac{|\sin nx|}{x}dx.
%\end{equation*}
%Since
%\begin{equation*}
%\begin{split}
%\int_0^{\pi}\frac{|\sin nx|}{x}dx&=\int_0^{nx}\frac{|\sin x|}{x}dx=\sum_{k=0}^{n-1}\int_{k\pi}^{(k+1)\pi}\frac{|\sin x|}{x}dx\\
%&=\sum_{k=0}^{n-1}\int_0^{\pi}\frac{|\sin(x+k\pi)|}{x+k\pi}dx=\int_0^{\pi}\sin x\sum_{k=0}^{n-1}\frac{1}{x+k\pi}dx,
%\end{split}
%\end{equation*}
%we have
%\begin{equation*}
%\int_0^{\pi}\frac{|\sin nx|}{x}dx=\int_0^{\pi}\frac{\sin x}{x}dx+\int_0^{\pi}\sin x\sum_{k=1}^{n-1}\frac{1}{x+k\pi}dx.
%\end{equation*}
%In $0\leq x\leq\pi$,
%\begin{equation*}
%\sum_{k=1}^{n-1}\frac{1}{x+k\pi}\leq\frac{1}{\pi}\sum_{k=1}^{n-1}\frac 1k,
%\end{equation*}
%and
%\begin{equation*}
%\sum_{k=1}^{n-1}\frac 1k\leq1+\int_1^{n-1}\frac 1xdx=1+\log(n-1),
%\end{equation*}
%the inequality becomes an equality when $n=2$. Therefore
%\begin{equation*}
%\int_0^{\pi}\sin x\sum_{k=1}^{n-1}\frac{1}{x+k\pi}dx\leq\frac{2}{\pi}[1+\log(n-1)],
%\end{equation*}
and
\begin{equation}\label{expression3}
\int_0^{\pi}\frac{|\sin nx|}{\tan \frac x2}dx\leq2\int_0^{\pi}\frac{\sin x}{x}dx+\frac{4}{\pi}[1+\log(n-1)].
\end{equation}%\end{itemize}
Together with \eqref{expression1}, \eqref{expression2} and \eqref{expression3}, we obtain
\begin{equation*}
%\begin{split}
\frac{1}{2\pi}\int_{-\pi}^{\pi}\Big|\frac{\sin(n+\frac12)x}{\sin\frac x2}\Big|dx
\leq\frac{2}{\pi}\int_0^{\pi}\frac{\sin x}{x}dx+\frac{4}{\pi^2}[1+\log(n-1)]+\frac{2}{\pi}.\\
%
%\end{split}
\end{equation*}
Thus,
\begin{equation}\label{ineq:eta}
\frac{1}{2\pi}\int_{-\pi}^{\pi}|D_n(x)|dx\leq\frac{4}{\pi^2}\log(n-1)+\eta,
\end{equation}
where
\begin{equation*}
\eta:=\frac{4}{\pi^2}+\frac{2}{\pi}+\dfrac{2\int_0^{\pi}\frac{\sin x}{x}dx}{\pi}.
\end{equation*}
The value of $\eta$ can be calculated by MATLAB with the function \verb"quadgk" as
$$\eta:=\frac{4}{\pi^2}+\frac{2}{\pi}+\dfrac{2\int_0^{\pi}\frac{\sin x}{x}dx}{\pi}=2.220884\ldots.$$
It is not difficult to see inequality \eqref{ineq:eta} is sharper than inequality \eqref{ineq:R1} for $n\geq 2$.
\end{proof}

\end{document}